\documentclass[11pt]{article}
\usepackage{amssymb,amsthm,amsmath,latexsym,geometry,tikz,setspace,enumerate}
\newtheorem{theorem}{Theorem}

\newtheorem{conjecture}[theorem]{Conjecture}
\newtheorem{claim}{Claim}
\newtheorem{case}{Case}

\begin{document}

\onehalfspace

\title{Induced Matchings in Graphs of Bounded Maximum Degree}

\author{Felix Joos\thanks{Institut f\"{u}r Optimierung und Operations Research, 
Universit\"{a}t Ulm, Ulm, Germany,
e-mail: \texttt{felix.joos@uni-ulm.de}}
}

\date{}

\maketitle

\vspace{-0.5cm}

%\begin{center}
%Institut f\"{u}r Optimierung und Operations Research, 
%Universit\"{a}t Ulm, Ulm, Germany\\
%\texttt{felix.joos@uni-ulm.de}
%\end{center}

\begin{abstract}
For a graph $G$, let $\nu_s(G)$ be the induced matching number of $G$.
We  prove that
$\nu_s(G) \geq \frac{n(G)}{(\lceil\frac{\Delta}{2}\rceil+1) (\lfloor\frac{\Delta}{2}\rfloor+1)}$
for every graph of sufficiently large maximum degree $\Delta$	 and without isolated vertices.
This bound is sharp.
Moreover, there is polynomial-time algorithm which computes induced matchings of size as stated above.
\end{abstract}


\section{Introduction}

For a graph $G$, a set $M$ of edges is an \emph{induced matching} of $G$ if
no two edges in $M$ have a common endvertex and no edge of $G$ joins two edges of $M$.
%if $M$ induces a $1$-regular graph in $G$; 
%that is, there is no edge joining two edges of $M$.
The maximum number of edges that form an induced matching in $G$ is the {\it strong matching number $\nu_s(G)$ of $G$}. 
We denote by $\Delta(G)$ the maximum degree of graph $G$
and let $n(G)=|V(G)|$ and $m(G)=|E(G)|$. 

In contrast to the well known matching number $\nu(G)$, which can be computed in polynomial time \cite{ed},
it is NP-hard to determine the strong matching number even in bipartite subcubic graphs \cite{ca1,lo,stva}.
%Moreover, it turned out to be difficult to give tight lower bounds for $\nu_s(G)$.
In fact, the strong matching number is even hard to approximate in restricted graphs classes as for example regular bipartite graphs \cite{dadelo}.
 
%If $G$ is a graph of maximum degree $\Delta$ and minimum degree $\delta$, respectively, 
To the best of my knowledge,
the only known bound in terms of the order and the maximum degree
for $\nu_s(G)$ is obtained by the following simple observation \cite{zito}.
Let $G$ be a graph without isolated vertices.
There are at most $2\Delta(G)^2-2\Delta(G)+1$ many edges in distance at most $1$ from $e$ including $e$
and $m(G)\geq \frac{1}{2} n(G)$.
Thus a simple greedy algorithm implies
\begin{align*}%\label{simple bound}
	\nu_s(G)\geq \frac{ n(G)}{2(2\Delta(G)^2-2\Delta(G)+1)},
\end{align*}
which is far away from being sharp if $G\not=K_2$.

It seems that the different behavior of $\nu(G)$ and $\nu_s(G)$ transfers to the corresponding partitioning problems.
The chromatic index $\chi'$ seems much simpler than the strong chromatic index $\chi_s'$, 
defined as the minimum number of induced matchings one needs to partition the edge set.
While for $\chi'(G)$ Vizing's Theorem always gives $\chi'(G)\in\{\Delta(G),\Delta(G)+1\}$ \cite{vi},
no comparable result holds for the strong chromatic index.
 
A trivial greedy algorithm ensures $\chi_s'(G)\leq 2\Delta(G)^2-2\Delta(G)+1$.
Erd{\H{o}}s and Ne{\v{s}}et{\v{r}}il \cite{fashgytu2} conjectured $\chi_s'(G)\leq \frac{5}{4}\Delta(G)^2$,
which would  be best possible for even $\Delta$
because
equality holds for the graph %$C_5^*$ 
obtained from the $5$-cycle by replacing every vertex by an independent set of order $\frac{\Delta}{2}$.
The best general result in this direction is due to Molloy and Reed, 
who proved that $\chi_s'(G)\leq 1.998\Delta(G)^2$ for sufficiently large maximum degree \cite{more}.
Thus Erd{\H{o}}s and Ne{\v{s}}et{\v{r}}il's conjecture is widely open and it is even unknown which technique is suitable to improve Molloy and Reed's result substantially. 

%This paper explains the behavior of the order 
In this paper I provide more insight concerning the behavior of induced matchings by improving the known lower bounds on $\nu_s(G)$ 
to a sharp lower bound
provided that the maximum degree is sufficiently large.
%Moreover, we state a conjecture for the general case.	
%For all proven lower bounds in this paper there are polynomial-time algorithms which compute an induced matching of the guaranteed size.
%In the next section we state our results in detail and prove them in the remaining sections.

%In \cite{zito} Zito proved that a graph $G$ with maximum degree $\Delta$ and without isolated vertices has an induced matching
%of size at least $\frac{n(G)}{4\Delta^2-4\Delta +2}$.
%We improve this bound for large $\Delta$.

\begin{theorem}\label{result max degree}
There is an integer $\Delta_0$ such that for every graph $G$ of maximum degree $\Delta$ at least $\Delta_0$ and without isolated vertices,
\begin{align*}%\label{eq max degree}
\nu_s(G) \geq \frac{n(G)}{\left(\lceil\frac{\Delta}{2}\rceil+1\right) \left(\lfloor\frac{\Delta}{2}\rfloor+1\right)}
\end{align*}
holds.
\end{theorem}
\noindent
The following construction shows that the bound in Theorem \ref{result max degree} is sharp.
Let $\Delta$ be an integer at least $3$ and
let the graph $H_1$ arise from the complete graph on $\left\lceil \frac{\Delta}{2} \right \rceil+1$ vertices by
attaching at each vertex $\left \lfloor \frac{\Delta}{2} \right\rfloor$ pendant vertices.
Let $H_2$ arise from the complete graph on $\left\lfloor \frac{\Delta}{2} \right \rfloor+1$ vertices by
attaching at each vertex $\left \lceil \frac{\Delta}{2} \right\rceil$ pendant vertices.
It follows that $\nu_s(H_i)=	1$ and $n(H_i)= \left(\left\lceil \frac{\Delta}{2} \right \rceil+1\right) \left(\lfloor\frac{\Delta}{2}\rfloor+1\right)$; that is, the bound of Theorem \ref{result max degree} is sharp.
Note that $H_1=H_2$ if $\Delta$ is even.
%Moreover, the proof of Theorem \ref{result max degree} also shows that $H_1$ and $H_2$ are the only graphs that attain equality.

For the sake of simplicity I do not try to optimize the constant $\Delta_0$ intensively.
We show  Theorem  \ref{result max degree} for $\Delta_0=1000$ but
with some more effort one can lower the bound down to $200$.

In \cite{jorasa} the same bound as in Theorem \ref{result max degree} is already shown by a simple inductive argument for graphs of girth at least $6$.
Hence one might ask whether the bound in Theorem \ref{result max degree} can be improved for graphs of large girth to $\frac{n(G)}{\Delta^c}$ for some $c<2$.
However, this is not the case.
By a result of Bollob{\'a}s \cite{bo},
for every $g \geq 3$ and $\Delta \geq 6$, there is a graph $H'$ 
of maximum degree $\lfloor\frac{\Delta}{2}\rfloor$, girth at least $g$, and independence number at most $\frac{4 \log \Delta}{\Delta}n(H')$.
Let $H$ arise from $H'$ by attaching to each vertex $\lceil\frac{\Delta}{2}\rceil$ many pendant vertices.
Note that $\nu_s(H)\leq \frac{4 \log \Delta}{\Delta}n(H')$ and $n(H)= \lceil\frac{\Delta}{2}\rceil n(H')$.
Thus $\nu_s(H) \leq \frac{8 \log \Delta}{\Delta^2}n(H)$ and the bound of Theorem \ref{result max degree} can only be improved by a $O(\log \Delta)$-factor.

Since the proof of Theorem \ref{result max degree} is constructive,
it is easy to derive a polynomial-time algorithm,
which computes an induced matching of size as guaranteed in Theorem \ref{result max degree}.

We use standard notation and terminology.
For a graph $G$, let $V(G)$ and $E(G)$ denote its vertex set and edge set, respectively.
For a vertex $v$, let $d_G(v)$ be its degree, let $N_G(v)$ be the set of neighbors of $v$, and
let $N_G[v]=N_G(v)\cup \{v\}$.
If the corresponding graph is clear from the context, we only write $d(v)$, $N(v)$ and $N[v]$, respectively. 
A set $I$ of vertices of $G$ is independent if there is no edge joining two vertices in $I$.
%Two edges $e$ and $f$ are independent if they do not share a common vertex and there is no edge of $G$ that is adjacent to $e$ and $f$.

\section{Proof of Theorem \ref{result max degree}}
We prove the theorem for $\Delta_0=1000$.
Let $G$ be a graph with maximum degree $\Delta$ at least $\Delta_0$ and without isolated vertices.
For a contradiction, we assume that $G$ is a counterexample such that 
\begin{enumerate}[(1)]
	\item $\nu_s(G)$ is minimum and
	\item subject to (1), the order of $G$ is maximum. 
\end{enumerate}
Since $\nu_s(G)\geq \frac{n(G)}{2\Delta^2}$, the graph $G$ is well-defined.

The choice of $G$ implies that if $v$ is a vertex of $G$ that is adjacent to a vertex of degree~$1$,
then $d(v)=\Delta$ because adding new vertices to $G$ and joining them to $v$ does not increase $\nu_s(G)$ but the order of $G$.

For some calculations it might help to know that $\frac{\Delta^2}{4}+\Delta+\frac{3}{4}
\leq \left(\lceil\frac{\Delta}{2}\rceil+1\right) \left(\lfloor\frac{\Delta}{2}\rfloor+1\right)$.

\begin{claim}\label{d1}
For every edge $uv$ of $G$, we have $d(u)+d(v) > \frac{\Delta}{4}$.
\end{claim}
\begin{proof}[Proof of Claim \ref{d1}]
For a contradiction, we assume that there is an edge $uv$ such that
$d(u)+d(v) \leq \frac{\Delta}{4}$.
Let $S=N[u]\cup N[v]$ and
let $I$ be the set of all isolated vertices of $G-S$. 
Let $G'=G-S-I$.
Since $\nu_s(G)\geq \nu_s(G')+1$,
the choice of $G$ implies $\nu_s(G') \geq \frac{n(G')}{\left(\lceil\frac{\Delta}{2}\rceil+1\right) \left(\lfloor\frac{\Delta}{2}\rfloor+1\right)}$.

By using the assumption $d(u)+d(v) \leq \frac{\Delta}{4}$,
we conclude $|S|+|I|\leq \left(\frac{\Delta}{4}-2\right)\Delta+2< \left(\lceil\frac{\Delta}{2}\rceil+1\right) \left(\lfloor\frac{\Delta}{2}\rfloor+1\right)$.
Therefore, $uv$ together with a maximum induced matching of $G'$ is an induced matching of $G$
of size at least $\frac{n(G)}{\left(\lceil\frac{\Delta}{2}\rceil+1\right) \left(\lfloor\frac{\Delta}{2}\rfloor+1\right)}$,
which contradicts the choice of $G$.
\end{proof}

\begin{claim}\label{d2}
Every vertex $v$ of $G$ is adjacent to at most $\frac{3}{4}\Delta$ many vertices of degree at most $9$.
\end{claim}
\begin{proof}[Proof of Claim \ref{d2}]
Choose $v$ such that the number of neighbors of degree at most $9$ is maximal.
Say $v$ has $\alpha\Delta$ many such neighbors.
For a contradiction, we assume that $\alpha>\frac{3}{4}\Delta$.
Let $u\in N(v)$ be of degree at most $9$.
As above, let $S=N[u]\cup N[v]$ and
%Obviously, $|S|\leq \Delta+10$.
let $I$ be the set of all isolated vertices of $G-S$. 
Let $G'=G-S-I$.
%Note that $|I|$ is less than the number of vertices with exactly distance one from $S$.
By Claim \ref{d1}, every vertex in $I$ that is adjacent to a vertex of degree at most $9$, has degree at least $10$.
Thus there are at most $(1-\alpha)\Delta+8$ many vertices in $S$ that are adjacent to vertices in $I$ of degree at most $9$.
Hence there are at most $\alpha(1-\alpha)\Delta^2+8\Delta$ many vertices in $I$ of degree at most $9$.
Furthermore, at most $8\alpha \Delta$ edges join vertices in $I$ and vertices in $N(v)\setminus \{u\}$ such that the vertices in $N(v)\setminus \{u\}$ have degree at most $9$.
Since $\alpha(1-\alpha)+ \frac{1}{10}(1-\alpha)^2< 0.22$, this implies
\begin{align*}
	|I|&\leq \alpha(1-\alpha)\Delta^2+8\Delta + \frac{1}{10}\left((1-\alpha)^2\Delta^2 +8\alpha\Delta \right)\\
	&< 0.22\Delta^2+9\Delta\\
	&\leq \frac{\Delta^2}{4}-9.
\end{align*}
Since $|S|\leq \Delta + 9$,
we obtain
$$|I|+|S|< \left(\left\lceil\frac{\Delta}{2}\right\rceil+1\right) \left(\left\lfloor\frac{\Delta}{2}\right\rfloor+1\right). $$
Again, the edge $uv$ together with a maximum induced matching of $G'$ is an induced matching of $G$
of size at least $\frac{n(G)}{\left(\lceil\frac{\Delta}{2}\rceil+1\right) \left(\lfloor\frac{\Delta}{2}\rfloor+1\right)}$,
which contradicts the choice of $G$.
\end{proof}

\noindent
Let $f: V(G)\rightarrow \mathbb{R}$ be such that
\begin{align*}
	f(v)= \sum_{w\in N(v):\ d(w)\not= \Delta}\frac{1}{d(w)}.
\end{align*}

\begin{claim}\label{d3}
If a vertex $v$ of $G$ is not adjacent to a vertex of degree $1$,
then $f(v)\leq \frac{2}{5}\Delta$.
\end{claim}
\begin{proof}[Proof of Claim \ref{d3}]
Let $v$ be a vertex that is not adjacent to a vertex of degree $1$.
By Claim \ref{d2}, the vertex $v$ has at most $\frac{3}{4}\Delta$ neighbors of degree at most $9$, which contribute to $f(v)$ at most $\frac{1}{2}$ each;
all remaining neighbors contribute at most $\frac{1}{10}$ each.
Thus $f(v)\leq \frac{1}{2}\cdot\frac{3}{4}\Delta+ \frac{1}{10}\cdot\frac{1}{4}\Delta = \frac{2}{5}\Delta$.
\end{proof}
\noindent
For the rest of the proof,
let $v\in V(G)$ be chosen such that $f(v)$ is maximal.

\begin{case}
$v$ is adjacent to a vertex of degree $1$.
\end{case}
Recall that this implies $d(v)=\Delta$.
Let $u\in N(v)$ be a vertex of degree~$1$.
As before, we want to combine $uv$ with a maximum induced matching of $G'=G-(N[v]\cup I)$ to obtain a contradiction, where $I$ are the isolated vertices of $G-N[v]$.

If $z\in I$ has degree $d<\Delta$, then $z$ contributes exactly $d$ times exactly $\frac{1}{d}$ to $f(w)$ for some $w\in N(v)$;
that is, the total contribution to $\sum_{w\in N(v)}f(w)$ is $1$.
Since no vertex in $I$ is adjacent to $u$, 
there is no vertex $z\in I$ such that $d(z)=\Delta$.
This implies that
\begin{align}\label{sum f}
	|I|\leq \sum_{w\in N(v)}f(w).
\end{align}
%because every vertex in $I$ contributes exactly $1$ to $\sum_{w\in N(v)}f(w)$.

Let $N_1$ and $N_{\Delta}$ be the set of vertices in $N(v)$ of degree
$1$ and $\Delta$, respectively.
Let $N_s$ be the set of vertices in $N(v)\setminus (N_1 \cup N_\Delta)$ of small degree, say such that their degree is between $2$ and $\frac{\Delta}{8}$.
Let $N_\ell=N(v)\setminus(N_1\cup N_s \cup N_\Delta)$, and let $n_1=|N_1|$, $n_s=|N_s|$, $n_\ell=|N_\ell|$, and $n_\Delta=|N_\Delta|$. 

Since all vertices in $N_s\cup N_\ell$ do not have degree $\Delta$ and by the choice of $G$,
they are not adjacent to a vertex of degree $1$.
If $w \in N_1$, then $f(w)=0$ and $w$ contributes $1$ to $f(v)$.
If $w \in N_s$, then by Claim \ref{d1}, we conclude $f(w)\leq 1$, and the contribution of $w$ to $f(v)$ is at most $\frac{1}{2}$.
If $w \in N_\ell$, then by Claim \ref{d3} and the choice of $v$, we obtain $f(w)\leq \min\left\{\frac{2}{5}\Delta,f(v)\right\}$ and the contribution of $w$ to $f(v)$ is at most $\frac{8}{\Delta}$.
If $w \in N_\Delta$, then $f(w)\leq f(v)$ and $w$ contributes nothing to $f(v)$.
These observations imply both
\begin{align*}
	f(v) \leq \frac{8}{\Delta}n_\ell + \frac{1}{2}n_s + n_1
\end{align*}
and, by using (\ref{sum f}),
\begin{align*}
	|I|\leq f(v) n_\Delta + \min\left\{\frac{2}{5}\Delta,f(v)\right\}n_\ell+ n_s.
\end{align*}
In order to prove that $|I|\leq \lceil \frac{\Delta}{2}\rceil \lfloor \frac{\Delta}{2} \rfloor$,
we show that
\begin{align}\label{eq I}
	f' n_\Delta + \min\left\{\frac{2}{5}\Delta,f'\right\}n_\ell+ n_s \leq \left\lceil \frac{\Delta}{2}\right\rceil \left\lfloor \frac{\Delta}{2}\right \rfloor,
\end{align}
under the condition that $n_1,n_s,n_\ell,n_\Delta$ are non-negative integers and $n_1+n_s+n_\ell+n_\Delta=\Delta$
where 
\begin{align}\label{eq fv}
	f'=\frac{8}{\Delta}n_\ell + \frac{1}{2}n_s + n_1.
\end{align}
Let $i(n_1,n_s,n_\ell,n_\Delta)=f' n_\Delta + \min\left\{\frac{2}{5}\Delta,f'\right\}n_\ell+ n_s$.
Obviously, $|I|\leq i(n_1,n_s,n_\ell,n_\Delta)$.

%and hence $|I|\leq \lceil \frac{\Delta}{2}\rceil \lfloor \frac{\Delta}{2} \rfloor$.

%\begin{claim}\label{d4}
%$f(v)\geq \frac{2}{5}\Delta+1$.
%\end{claim}
%\begin{proof}[Proof of Claim \ref{d4}]
%\end{proof}

Inequality (\ref{eq fv}) implies $n_s+n_1\geq f'-8$.
Thus $n_\ell+n_\Delta = \Delta- n_1-n_s \leq \Delta-f'+8$ and
hence, by (\ref{eq I}), we obtain
\begin{align*}
	i(n_1,n_s,n_\ell,n_\Delta)\leq f'(\Delta-f'+8) +\Delta.
\end{align*}
If $f'\leq \frac{2}{5}\Delta+8$, then this implies that 
$i(n_1,n_s,n_\ell,n_\Delta)\leq \frac{6}{25}\Delta^2+ \frac{24}{5}\Delta \leq
\frac{\Delta^2}{4}-1$, which implies the desired result.

Thus we may assume that $f'\geq \frac{2}{5}\Delta+8$.
Suppose $n_\ell \geq 1$ and hence $n_\Delta \leq \Delta-1$.
This implies that 
\begin{align*}
	i(n_1,n_s,n_\ell-1,n_\Delta+1)-i(n_1,n_s,n_\ell,n_\Delta)
	&\geq - \frac{8}{\Delta}n_\Delta- \frac{2}{5}\Delta + \left(f'- \frac{8}{\Delta}\right)\cdot 1\\
	&\geq - \frac{8}{\Delta}(\Delta-1)- \frac{2}{5}\Delta + \frac{2}{5}\Delta+8- \frac{8}{\Delta}\\
	&= 0.
\end{align*}
Hence, we may assume that $n_\ell=0$.
%Since decreasing $n_\ell$ by $1$ and increasing $n_\Delta$ by $1$, 
%does not decrease
%$f(v) n_\Delta + \frac{2}{5}\Delta n_\ell+ n_s$, we may assume that $n_\ell=0$.

%Furthermore, the inequality $n_\ell+n_\Delta\leq \Delta-f(v)+8$ implies $n_\Delta\leq \frac{3}{5}\Delta$.
Furthermore, we may assume that $n_\Delta \geq 2$; otherwise, by using $f',n_s\leq \Delta$, we conclude $i(n_1,n_s,n_\ell,n_\Delta) \leq 2\Delta$.
Suppose $n_s\geq 1$. Thus 
\begin{align*}
	i(n_1+1,n_s-1,n_\ell,n_\Delta)-i(n_1,n_s,n_\ell,n_\Delta)\geq \frac{1}{2} \cdot 2 - 1\geq 0.
\end{align*}
Therefore, we may assume that $n_s=0$.
Thus $n_1=\Delta-n_\Delta$ and (\ref{eq fv}) implies that $f'=n_1$.
By using (\ref{eq I}), we conclude
$$i(n_1,n_s,n_\ell,n_\Delta)=n_\Delta(\Delta-n_\Delta)\leq \left\lceil \frac{\Delta}{2}\right\rceil \left\lfloor \frac{\Delta}{2}\right \rfloor.$$

Therefore, $|N[v]|+|I|\leq \left(\lceil\frac{\Delta}{2}\rceil+1\right) \left(\lfloor\frac{\Delta}{2}\rfloor+1\right)$
and the edge $uv$ together with a maximum induced matching of $G'$ yields
$\nu_s(G)\geq \frac{n(G)}{\left(\lceil\frac{\Delta}{2}\rceil+1\right) \left(\lfloor\frac{\Delta}{2}\rfloor+1\right)}$,
which is a contradiction to our choice of $G$.

\begin{case}
$v$ is not adjacent to a vertex of degree $1$.
\end{case}
Let $u\in N(v)$ such that $d(u)$ is minimal.
Let $S=N[u]\cup N[v]$ and $G'=G-S-I$ where $I$ is the set of isolated vertices of $G-S$.
%Suppose $d(u)\geq 10$ and hence $f(v)\leq \frac{\Delta}{10}$.
%It follows that $f(v)\leq \frac{\Delta}{11}$.
By double counting the edges between $S$ and $I$, it is straightforward to see that $I$ contains at most $2\Delta$ vertices of degree $\Delta$.
Thus similarly as in (\ref{sum f}), we conclude that
\begin{align}\label{eq 4}
	|I|\leq \sum_{w\in S\setminus\{u,v\}}f(w) + 2\Delta. 
\end{align}
%because now there might be a vertex $w$ with $d(w)=\Delta$ and $w\in I$.
If $d(u)\geq 10$, 
then $f(v) \leq \frac{\Delta}{10}$.
Thus $|I|\leq \frac{\Delta^2}{5}+2\Delta$ and hence $|S|+|I|\leq \frac{\Delta^2}{4}$. %, which is, by the same argumentation as above, a contradiction.
Therefore, $uv$ together with a maximum induced matching of $G'$ yields
$\nu_s(G)> \frac{n(G)}{\left(\lceil\frac{\Delta}{2}\rceil+1\right) \left(\lfloor\frac{\Delta}{2}\rfloor+1\right)}$,
which is a contradiction to our choice of $G$.

Thus we may assume that $d(u)\leq 9$ and hence trivially $\sum_{w\in N(u)\setminus \{v\}}f(w)\leq 8\Delta$ and $|S|\leq \Delta+9$.
%By Claim \ref{d3}, $f(v)< \frac{2}{5}\Delta$.
Let $N_s$ be the set of neighbors of $v$ of degree at most $\frac{\Delta}{8}$, let $N_\ell=N(v)\setminus N_s$,
and let $\alpha= \frac{|N_s|}{\Delta}$ and hence $N_\ell \leq (1-\alpha)\Delta$.

The contribution of the vertices in $N_s$ to $f(v)$ is at most $\frac{\alpha\Delta}{2}$.
Using Claim \ref{d1}, we conclude that $f(w)\leq 1$ for $w\in N_s$.
The contribution of the vertices in $N_\ell$ to $f(v)$ is at most $8$
and $f(w)\leq f(v)$ for $w\in N_\ell$ by the choice of $v$.
This implies that $f(v)\leq \frac{\alpha\Delta}{2}+8$.
Note that $(1-\alpha)\frac{\alpha}{2}\leq \frac{1}{8}$.
Moreover, by (\ref{eq 4}), we obtain
\begin{align*}
	|I|& \leq  \sum_{w\in N(v)\setminus \{u\}}f(w)+\sum_{w\in N(u)\setminus \{v\}}f(w)+2\Delta\\
	& \leq \sum_{w\in N(v)\setminus \{u\}: w\in N_\ell}f(w)+\sum_{w\in N(v)\setminus \{u\}: w\in N_s}f(w)+8\Delta+2\Delta\\
	& \leq  (1-\alpha)\Delta f(v) + \alpha\Delta + 10\Delta\\
		&\leq (1-\alpha)\Delta \left(\frac{\alpha\Delta}{2}+8\right) + 11\Delta\\
	&\leq  \frac{\Delta^2}{4}-2\Delta.
\end{align*}
Thus $|I|+|S|\leq  \frac{\Delta^2}{4}$.
Therefore, $uv$ together with a maximum induced matching of $G'$ yields
$\nu_s(G)> \frac{n(G)}{\left(\lceil\frac{\Delta}{2}\rceil+1\right) \left(\lfloor\frac{\Delta}{2}\rfloor+1\right)}$,
which is the final contradiction.
\hfill $\Box$

\section{Graphs with Small Maximum Degree
}
Let $C_5^2$ be the graph obtained from the $5$-cycle by replacing every vertex by an independent set of order $2$ and
let $K_{3,3}^+$ be the graph obtained from the $5$-cycle by replacing the vertices by independent sets of orders $1,1,1,2$, and $2$, respectively.
Note that the graph $K_{3,3}^+$ can also be obtained from a $K_{3,3}$ by subdividing one edge once.
The graphs $C_5^2$ and $K_{3,3}^+$ show that Theorem \ref{result max degree} is not true for graphs of maximum degree $3$ or $4$.
However, I conjecture that these graphs are the only exceptions.

\begin{conjecture}\label{conj general}
If connected graph $G\notin\{C_5^2,K_{3,3}^+\}$ with maximum degree $\Delta\geq 3$,
then
\begin{align*}
\nu_s(G) \geq \frac{1}{\left(\lceil\frac{\Delta}{2}\rceil+1\right) \left(\lfloor\frac{\Delta}{2}\rfloor+1\right)}n(G).	
\end{align*}
\end{conjecture}
\noindent
Note that for $\Delta=3$, a result in \cite{jorasa}, and 
for $\Delta \geq 1000$, Theorem \ref{result max degree} implies Conjecture~\ref{conj general}.

\end{document}